\documentclass[11pt, twoside]{amsart}

\usepackage{amsmath,amssymb,amsthm,enumitem,stmaryrd,mathrsfs}

\usepackage{thm-restate,hyperref}

\usepackage[pdftex]{graphicx}

\usepackage[top=1.5in, bottom=1.5in, left=1.5in, right=1.5in]{geometry}

\usepackage{tikz}

\allowdisplaybreaks

\theoremstyle{plain}
\newtheorem{theorem}{Theorem}[section]
\newtheorem{lemma}[theorem]{Lemma}
\newtheorem{cor}[theorem]{Corollary}
\newtheorem{prop}[theorem]{Proposition}

\newtheorem*{lemma*}{Lemma}
\newtheorem*{cor*}{Corollary}
\newtheorem*{theorem*}{Theorem}

\theoremstyle{definition}

\newtheorem{problem*}{Problem}

\newtheorem{definition}[theorem]{Definition}

\theoremstyle{remark}

\newtheorem*{fact*}{Fact}
\newtheorem*{remark}{Remark}

\newtheorem*{question}{Question}

\let\oldproofname=\proofname
\renewcommand{\proofname}{\rm\bf{\oldproofname}}

\newcommand{\R}{\mathbb R}
\newcommand{\Q}{\mathbb Q}
\newcommand{\Z}{\mathbb Z}
\newcommand{\C}{\mathbb C}

\newcommand{\F}{\mathbb F}

\newcommand{\RP}{\mathbb R\text{P}}

\newcommand{\phii}{\varphi}
\newcommand{\ep}{\varepsilon}

\makeatletter
\providecommand*{\twoheadrightarrowfill@}{%
  \arrowfill@\relbar\relbar\twoheadrightarrow
}
\providecommand*{\twoheadleftarrowfill@}{%
  \arrowfill@\twoheadleftarrow\relbar\relbar
}
\providecommand*{\xtwoheadrightarrow}[2][]{%
  \ext@arrow 0579\twoheadrightarrowfill@{#1}{#2}%
}
\providecommand*{\xtwoheadleftarrow}[2][]{%
  \ext@arrow 5097\twoheadleftarrowfill@{#1}{#2}%
}
\makeatother

\newcommand{\op}[1]{\operatorname{{#1}}}

\newcommand{\mc}[1]{\mathcal{{#1}}}

\begin{document}
    \title{The Rational Homotopy of Stable $C_p$-Smoothings}
    \author{Oliver H. Wang}
    \address{Department of Mathematics, The University of Virginia, Charlottesville, VA 22904, USA}
\email{dfh3fs@virginia.edu}
	\maketitle

\begin{abstract}
Smooth structures on high dimensional manifolds are classified by maps to the infinite loop space $TOP/O$.
The homotopy groups of this space are known to be finite.
Given a compact Lie group $G$, this space can be regarded as an equivariant infinite loop space and equivariant maps from a locally linear, high dimensional $G$-manifold to $TOP/O$ classify stable $G$-smoothings.
We compute the equivariant homotopy groups $\pi_V^{C_p}TOP/O\otimes\Q$ where $C_p$ denotes the cyclic group of order $p$.
By applying our methods to the group $C_4$, we prove a Chern class analogue of Novikov's theorem on rational Pontryagin classes.
\end{abstract}

\section{Introduction}
Kirby--Siebenmann \cite{KirbySiebenmann} show that the infinite loop space $TOP/ O$ classifies smooth structures on high dimensional manifolds.
An application of Kervaire--Milnor's theorem on the finiteness of the group of homotopy spheres in high dimensions \cite{KervaireMilnor} shows that the homotopy groups of $TOP/ O$ are finite in dimensions at least $5$.
A separate analysis of the low dimension homotopy groups shows that $\pi_iTOP/ O=0$ for $i=0,1,2,4$ and $\pi_3 TOP/ O\cong\Z/2$.
Thus $\pi_i TOP/ O$ is finite.

This result has very interesting consequences.
First, an application of obstruction theory shows that $[X, TOP/ O]$ is finite for any finite CW-complex $X$.
Hence any compact high dimensional manifold has only finitely many smooth structures.
Additionally, the rational triviality of $TOP/O$ implies that $BO$ is rationally equivalent to $BTOP$.
This recovers Novikov's famous result that the rational Pontryagin classes of a manifold are topological invariants.

In this paper, we consider $ TOP/O$ as a $C_p$-equivariant infinite loop space and we determine the groups $[S^V, TOP/ O]^{C_p}$ rationally, where $V$ is a $C_p$-representation and $S^V$ denotes the representation sphere.
Equivariant homotopy classes of maps from a $G$-manifold to $ TOP/O$ also admit a geometric description which we review below.

\subsection{$G$-Smoothing Theory}
In \cite{LashofRothenberg}, Lashof and Rothenberg develop smoothing theory for $G$-manifolds.
We summarize their results here.

\begin{definition}\label{def: isotopy of smooth structures}
Let $X$ be a $G$-manifold.
A \emph{$G$-smoothing of $X$} is a pair $(Y,f)$ where $Y$ is a smooth $G$-manifold and $f:Y\rightarrow X$ is a $G$-homeomorphism.
Two $G$-smoothings $(Y_i,f_i)$, $i=0,1$ are \emph{isotopic} if there is a $G$-homeomorphism $\alpha:Y_0\times I\rightarrow X$ such that the following hold:
\begin{itemize}
\item For $t\in I$, $\alpha(-,t)$ is a $G$-homeomorphism,
\item $\alpha(-,0)=f_0$ and $f_1^{-1}\circ\alpha(-,1):Y_0\rightarrow Y_1$ is a $G$-diffeomorphism.
\end{itemize}
In this definition, $Y_0\times I$ denotes the product smooth $G$-manifold.
\end{definition}

\begin{remark}
Given two smooth $G$-manifolds $X$ and $Y$, their product $X\times Y$ can be given the structure of a smooth $G$-manifold in the standard way; as a smooth manifold it is just the product and $G$ acts diagonally.
In this case, we will say that $X\times Y$ is the product smooth $G$-manifold.
A very important subtlety in the theory of smooth $G$-manifolds is that there are smooth $G$-manifolds which are $G$-homeomorphic to $X\times I$ but which are not diffeomorphic to a product $Y\times I$ for any smooth $G$-manifold $Y$.
Therefore, it is important to specify in definition \ref{def: isotopy of smooth structures} above that $Y_0\times I$ is the product smooth $G$-manifold.
\end{remark}

Let $V$ be a finite dimensional $G$-representation.
Let $ TOP_G(V)$ denote the homeomorphisms of $V$ commuting with the $G$-action.
Similarly, let $ O_G(V)$ denote the orthogonal transformations of $V$ commuting with the $G$-action.
There is a $G$-space $BO_n(G)$ such classifying $n$-dimensional $G$-vector bundles.
This space has the property that
\[
BO_n(G)^H=\coprod_{V}BO_H(V)
\]
where the disjoint union is indexed by the $n$-dimensional $H$-representations.

There is also a $G$-space $B TOP_n(G)$ classifying $n$-dimensional $G$-microbundles.
The fixed sets are
\[
B TOP_n(G)^H=\coprod_{V}B TOP_H(V)
\]
where the disjoint union is indexed by $ TOP$-equivalence classes of $H$-representations.

\begin{remark}
When $H$ has odd order, \cite{MadsenRothenberg1} and \cite{HsiangPardon} show that $H$-representations are topologically equivalent if and only if they are isomorphic.
In particular, $BO_n(G)^H\rightarrow B TOP_n(G)^H$ is bijective on components if $H$ has odd order.
\end{remark}

Given an $n$-dimensional locally linear $G$-manifold $X$, one may take the tangent $G$-microbundle as one does in the non-equivariant case.
This is classified by a $G$-map $\tau:X\rightarrow B TOP_n(G)$.
\begin{theorem}\label{thm: Lashof--Rothenberg G-smoothing}
Let $X$ be an $n$-dimensional locally linear $G$-manifold such that $\dim X^H\neq4$ for any subgroup $H\le G$.
Then, isotopy classes of $G$-smoothings of $X$ are classified by $G$-homotopy classes of lifts
\[
\begin{tikzpicture}[scale=2]
\node (A) at (2,1) {$BO_n(G)$};
\node (B) at (0,0) {$X$};\node (C) at (2,0) {$BTOP_n(G)$};
\path[->] (A) edge (C) (B) edge node[above]{$\tau$} (C) (B) edge[dashed] (A);
\end{tikzpicture}.
\]
\end{theorem}

In the non-equivariant setting, homotopy classes of lifts can be studied using obstruction theory using cohomology valued in the homotopy of the fiber $ TOP(n)/ O(n)$.
However, the fixed sets of the map $ O_n(G)\rightarrow  TOP_n(G)$ are generally not connected so it does not make sense to map into a fiber.

\subsection{Stable $G$-Smoothings}
\begin{definition}\label{def: stable smoothing}
Let $X$ be a $G$-manifold.
A \emph{stable $G$-smoothing of $X$} consists of a representation $\rho$ and a smoothing $(Y,f)$ or $X\times\rho$.
Two stable $G$-smoothings $(Y_i,f_i,\rho_i)$, $i=0,1$, are \emph{stably isotopic} if there are representations $\sigma_i$ such that $\rho_0\oplus\sigma_0=\rho_1\oplus\sigma_1$ and $(Y_i\times\sigma_i,f_i\times\op{id}_{\sigma_i})$ are isotopic $G$-smoothings.
\end{definition}

Classification of stable $G$-smoothings is much nicer homotopically than the unstable case.
Define $BO(G):=\bigcup_n BO_n(G)$ and $B TOP(G):=\bigcup_n B TOP_n(G)$.
Lashof \cite{LashofStable} shows the following.

\begin{theorem}\label{thm: Lashof stable}
Let $X$ be a locally linear $G$-manifold.
Stable isotopy classes of stable $G$-smoothings of $X$ are in bijection with $G$-homotopy classes of lifts
\[
\begin{tikzpicture}[scale=2]
\node (A) at (2,1) {$BO(G)$};
\node (B) at (0,0) {$X$}; \node (C) at (2,0) {$B TOP(G)$};
\path[->] (A) edge (C) (B) edge node[above]{$\tau$} (C) (B) edge[dashed] (A);
\end{tikzpicture}.
\]
\end{theorem}

Once we stabilize, the fixed sets $BO(G)^H$ and $BTOP(G)^H$ are connected so $G$-homotopy classes of lifts are classified by $G$-homotopy classes of maps to a fiber $ TOP(G)/ O(G)$.

\subsection{Conventions}
Note that the underlying topological space of $BO(G)$ is $BO(G)^e=BO$.
Similarly, the underlying space of $BTOP(G)$ is $BTOP$.
We simply write $BO$ for the $G$-space $BO(G)$ and similarly for $BTOP$.
We use $O_G$ and $TOP_G$ to denote the group of stable automorphisms commuting with the group action.
In particular, $BO_G$ and $BTOP_G$ are the fixed sets of $BO$ and $BTOP$.
We adopt similar conventions for the self-homotopy equivalences $F$ and $F_G$, introduced below.

We write $C_p$ for the cyclic group of order $p$.
Definitions and results stated for a group $G$ hold for any finite group.

\subsection{Main Results}
Before stating the main results of this paper, we mention some related results in the literature.
Madsen--Rothenberg \cite{MadsenRothenbergEquivAut} study the $G$-spaces $F/ TOP$ and $F/PL$, where $F=\varinjlim_VF(V)$ and $F(V)$ denotes the self-homotopy equivalences of $S^V$ with $G$ acting by conjugation.
They show that \cite[Theorem 1.1]{MadsenRothenbergEquivAut}, for an odd prime $p$, there are isomorphisms
\[
\pi_i F_{C_p}/ TOP_{C_p}\cong L_i^{\langle-\infty\rangle}(C_p)\oplus L_i(e).
\]
They also compute $[X,F/ TOP]^{C_p}$ for a $C_p$-CW-complex $X$ after inverting $2$.

For an odd prime $p$, let $t$ denote the order of $2$ in $\F_p^\times$.
Let $\mc{E}_p$ denote the $\Q$-vector space with the following dimension
\[
\dim\mc{E}_p=\begin{cases}
\frac{p-1}{2t}&t\text{ odd}\\
0&t\text{ even}
\end{cases}.
\]
The vector space $\mc{E}_p$ encodes $\Q$-linear relations between certain algebraic numbers appearing in the Atiyah--Singer $G$-signature formula.
Our main results are the following.

\begin{restatable}{theorem}{MainTheoremOdd}\label{thm: main odd}
Then there are isomorphisms
\[
[S^V, TOP/ O]^{C_p}\otimes\Q\cong\begin{cases}\mc{E}_p&\dim V^{C_p}=1,2\\ \Q&\dim V^{C_p}\equiv 3\text{ mod }4\\0&\text{ otherwise}\end{cases}.
\]
\end{restatable}

\begin{restatable}{theorem}{MainTheoremEven}\label{thm: main even}
There are isomorphisms
\[
[S^V, TOP/ O]^{C_2}\otimes\Q\cong 0.
\]
\end{restatable}

In the case where $p$ is odd, the rational homotopy in degrees $1$ and $2$ appear because there are $C_p$-vector bundles over $S^2$ that are trivial topologically.
These vector bundles can be distinguished by the first Chern classes of their eigenbundles.
The rational homotopy in degrees $3$ modulo $4$ appear because there are two copies of $L(e)$ in $F_G/ TOP_G$ (note that the $L^{\langle-\infty\rangle}(G)$ term of Madsen--Rothenberg's isomorphism is \emph{unreduced}) and $BO_G$ can only cancel out one of them.
In the case $p=2$, these groups are rationally trivial, essentially because $L(C_2)$ and $BO_{C_2}$ are rationally equivalent.

The proof of Theorems \ref{thm: main odd} and \ref{thm: main even} use results of Ewing \cite{EwingSpheresAsFPSets}, Schultz \cite{SchultzSpherelike} and Cappell--Weinberger \cite{CappellWeinbergerSimpleAS} to understand $BO_{C_p}(V)\to B\widetilde{SPL}_{C_p}(V)$ rationally when $V$ is a free representation.
Then, we apply results of Madsen--Rothenberg \cite{MadsenRothenberg2} to understand the map of $C_p$-spaces $BO\to BPL$ rationally.
By \cite{MadsenRothenbergEquivAut}, this is sufficient for understanding the map $BO\to BTOP$ rationally.

\subsection{Acknowledgments}
The author would like to thank Shmuel Weinberger for helpful conversations.
The author would also like to thank Alexander Kupers for asking a question that motivated Section \ref{sec: chern classes}.
The author was supported by NSF Grant DMS-1839968.

\section{Reduction to the Fixed Set}
In this section, we show that $[X,TOP/O]^{C_p}$ is rationally isomorphic to $[X^{C_p}, TOP_{C_p}/ O_{C_p}]$.
In order to make sense of this statement, we must show that $ TOP/ O$ is an equivariant infinite loop space.
It appears that this fact is known, or at least expected, but we have not been able to locate a proof in the literature so we sketch one below.

\begin{prop}\label{prop: TOP/O equivariant infinite loop space}
The $G$-space $ TOP/ O$ is an equivariant infinite loop space.
\end{prop}
\begin{proof}
Costenoble--Waner \cite{CostenobleWaner} showed that $BF$ is an equivariant infinite loop space.
One can similarly show that $BO$ and $BTOP$ are also equivariant infinite loop spaces.
Since $TOP/O$ is the fiber of a map of equivariant infinite loop spaces, it is also an equivariant infinite loop space.
\end{proof}

As a consequence of Proposition \ref{prop: TOP/O equivariant infinite loop space}, $[X,TOP/O]^G$ is an abelian group for any $G$-CW-complex $X$.
Additionally, $ TOP_G/ O_G$ is an infinite loop space so $[X, TOP_G/ O_G]$ is an abelian group for any $G$-CW-complex $X$.

\begin{lemma}\label{lem: transfer}
Suppose $X$ is a $G$-CW-complex such that $G$ acts freely except at the basepoint.
Let $Y$ be a pointed $G$-space.
\end{lemma}

As an application, we have the following result.
\begin{prop}\label{prop: reduction to fixed set}
Suppose $X$ is a finite $C_p$-CW-complex.
Then the map $[X,TOP/O]^{C_p}\rightarrow [X^{C_p}, TOP_{C_p}/ O_{C_p}]$ given by restriction is rationally an isomorphism.
\end{prop}
\begin{proof}
Since $ TOP/ O$ is an equivariant infinite loop space, $[-, TOP/ O]^{C_p}$ is the $0$-th group of an $RO(C_p)$-graded cohomology theory.
In particular, there is an exact sequence
\[
[X/X^{C_p},TOP/O]^{C_p}\rightarrow[X,TOP/O]^{C_p}\rightarrow[X^{C_p}, TOP_{C_p}/ O_{C_p}]\rightarrow[X/X^{C_p}, \Omega TOP/ O]^{C_p}.
\]
If we note that all cells of $X/X^{C_p}$ are free except for the basepoint, we may apply equivariant obstruction theory and use that $TOP/O$ has finite homotopy groups to conclude that $[X/X^{C_p},TOP/O]^{C_p}$ is finite.
This finishes the proof.
\end{proof}
\begin{remark}
Proposition \ref{prop: reduction to fixed set} holds for a general finite group $G$ if the only isotropy subgroups of $X$ are $G$ and the trivial group.
\end{remark}

\section{The Case $p$ is Odd}
In this section, we consider the case $G=C_p$ where $p$ is an odd prime.
The homotopy groups of $F_{C_p}/TOP_{C_p}$ were studied in \cite{MadsenRothenbergEquivAut}.
They showed that, rationally, the homotopy groups are isomorphic to those of $BO\times BO_{C_p}$.
So rationally, there is the following commuting diagram.
\begin{equation}\label{diag: Psi}
\begin{tikzpicture}[scale=2]
\node (A) at (0,1) {$\pi_mF_{C_p}/O_{C_p}$};\node (B) at (2,1) {$\pi_mF_{C_p}/TOP_{C_p}$};
\node (C) at (0,0) {$\pi_mBO_{C_p}$};\node (D) at (2,0) {$\pi_mBO\times BO_{C_p}$};
\path[->] (A) edge (B) (A) edge (C) (B) edge (D) (C) edge node[above]{$\Psi$} (D);
\end{tikzpicture}
\end{equation}
The vertical arrows are rationally isomorphisms so it suffices to understand the bottom arrow.
One may hope that this can be identified with the standard inclusion but this is not the case.
The isomorphism of Madsen--Rothenberg is obtain through surgery theory and the bottom map does admit an explicit description in terms of characteristic classes.
These characteristic classes originate from the Atiyah--Singer $G$-signature theorem so, a priori, they are very complicated.
Fortunately, Ewing's work \cite{EwingSpheresAsFPSets, EwingSemifree} essentially computes the map in question.

\subsection{Representations and Normal $G$-Vector Bundles}
We say that a $G$-representation is free if $G$ acts freely in the complement of the origin and we say that a $G$-vector bundle is free if its fibers are free.
When we work with a fixed $G$-smoothing $(Y,f)$, we let $E$ denote the normal bundle of $f^{-1}(M)$.

If $V$ is a $G$-representation and $M$ is a space, we let $\ep_V$ denote the $G$-vector bundle $M\times V\rightarrow M$.

For smooth $G$-manifolds $X$, Atiyah--Singer define the $G$-signature which is valued in $\widetilde{RO}(G)$.
They give a formula for the $G$-signature in terms of characteristic classes of the tangent bundle and the normal $G$-vector bundle of the fixed set.
Hence, given a free $G$-vector bundle $E$ over a smooth manifold $M$, we may use their formula to define the $G$-signature of $E$.
We omit an explicit description of the formula as we only need the results stated below.

Let $\zeta$ denote a fixed primitive $p$-th root of unity.
Given a generator $g_0$ of $C_p$, a finite dimensional real $C_p$-representation $V$ decomposes into a sum of eigenspaces
\[
V\cong\R^i\oplus\bigoplus_{k=1}^{\frac{p-1}{2}}V_k
\]
where $g_0$ acts trivially on $\R^i$ and each $V_k$ is a complex vector space with $g_0$ acting via multiplication by $\zeta^k$.
In particular, the reduced representation ring $\widetilde{RO}(C_p)$ is isomorphic, as an abelian group, to $\Z^{\frac{p-1}{2}}$ and $BO_{C_p}\simeq BO\times\prod_{k=1}^{\frac{p-1}{2}}BU$.
If $(Y,f)$ is a $C_p$-smoothing of a locally linear $C_p$-manifold $X$ and if $M$ is a component of the fixed set of $X$, then $M$ has a normal $C_p$-vector bundle $\nu$.
A generator $g_0$ of $C_p$ determines an eigenbundle decomposition $\nu=\bigoplus_{k=1}^{\frac{p-1}{2}}\nu_k$.
If two $C_p$-smoothings are isotopic, the normal bundles of the preimages of $M$ must be isomorphic.

\subsubsection{Ewing's Relations}
In \cite{EwingSpheresAsFPSets}, Ewing studies the Chern classes of normal $C_p$-vector bundles of fixed sets of smooth $C_p$-actions on even dimensional spheres.
In this situation, the fixed set is a $2m$-dimensional rational homology spheres and the Atiyah--Singer $G$-signature theorem implies
\[
\sum_{k=1}^{\frac{p-1}{2}}\Phi_{m,k}c_m(\nu_k)=0
\]
where the $\Phi_{m,k}\in\Q(\zeta)$ are either totally real or purely imaginary, depending on the parity of $m$.

Let $\C(\zeta^k)$ denote the irreducible $C_p$-representation where $g_0$ acts via multiplication by $\zeta^k$.
Define a $\Q$-linear transformation $\Phi_m:\widetilde{RO}(C_p)_{(0)}\rightarrow \Q(\zeta+(-1)^m\zeta^{-1})$ by $\Phi\left(\sum_{k=1}^{\frac{p-1}{2}}a_k\C(\zeta^k)\right)=\sum_{k=1}^{\frac{p-1}{2}}\Phi_{m,k}a_k$.
Since $\Q(\zeta+(-1)^m\zeta^{-1})$ is always a $\frac{p-1}{2}$-dimensional vector space, we regard $\Phi_m$ as a linear transformation $\widetilde{RO}(C_p)_{(0)}\to\Q^{\frac{p-1}{2}}$.
The following is \cite[Theorem 1.1]{EwingSpheresAsFPSets}.

\begin{theorem}[Ewing]\label{thm: Ewing}
The linear transformation $\Phi_m$ is an isomorphism when $m\neq1$.
If $2$ has even order in $\F_p^\times$, then $\Phi_1$ is an isomorphism.
If $2$ has odd order $t$ in $\F_p^\times$, then $\dim_{\Q}\ker\Phi_1=\frac{p-1}{2t}$.
\end{theorem}

We henceforth make the abbreviation $\mc{E}_p:=\ker\Phi_1$.





\subsection{Computation of $\pi_* TOP_{C_p}/ O_{C_p}$}

In order to compute $\pi_mTOP_{C_p}/O_{C_p}$ rationally, work of Madsen--Rothenberg allows us to restrict our attention to $\pi_mTOP_{C_p}(V)/O_{C_p}(V)$ for a finite dimensional representation $V$ sufficiently large with respect to $m$.
Writing $V=W\oplus\R^i$ where $W$ is a representation with no trivial summands, Madsen--Rothenberg's work reduces the study of these homotopy groups to the study of lens space block bundles over $S^m$.
Surgery theory reduces the study of lens space block bundles to the study of $G$-signatures, allowing us to describe the map $\Psi$ in Diagram (\ref{diag: Psi}) using $\Phi$.
We now spell out the details of this procedure.
To simplify notation in what follows, we make the abbreviation
\[
X:=(\pi_{2m}BO\times\pi_{2m}BO\times\pi_{2m}BU^{\frac{p-1}{2}})_{(0)}.
\]

\begin{prop}\label{prop: LES for p odd}
Let $E'\oplus E$ represent a stable $C_p$-vector bundle over $S^{2m}$ such that $C_p$ acts trivially on $E'$ and freely away from the zero section on $E$.
Let $E_k$ denote the eigenbundle of $E$ corresponding to the representation $\C(\zeta^k)$.
By abuse of notation, let $\Phi_m(E)$ denote $\Phi_m\left(\sum_{k=1}^{\frac{p-1}{2}}c_m(E_k)\C(\zeta^k)\right)$.
There is an isomorphism $(\pi_{2m} BTOP_{C_p})_{(0)}\cong X$ such that the composition
\[
(\pi_{2m}BO_{C_p})_{(0)}\to(\pi_{2m}BTOP_{C_p})_{(0)}\xrightarrow{\cong} X
\]
may be identified with
\[
E'\oplus E\mapsto(a(E),b(E'),\Phi_m(E))
\]
where $a:(\pi_{2m}BO_{C_p})_{(0)}\to(\pi_{2m}BO)_{(0)}$ is a surjection and $b:\pi_{2m}BO_{(0)}\to\pi_{2m}BO_{(0)}$ is an isomorphism.
\end{prop}

\begin{proof}
We begin by reducing the problem from a stable topological setting to an unstable $PL$-setting.
Specifically, for fixed $m$ and $V$ a sufficiently large representation, there is the following diagram.
\[
\begin{tikzpicture}[scale=2]
\node (A) at (0,2) {$\pi_{2m}BO_{C_p}(V)$};\node (B) at (2,2) {$\pi_{2m}BPL_{C_p}(V)$};\node (C) at (4,2) {$X$};
\node (D) at (2,1) {$\pi_{2m}BPL_{C_p}$};
\node (E) at (0,0) {$\pi_{2m}BO_{C_p}$};\node (F) at (2,0) {$\pi_{2m}BTOP_{C_p}$};
\path[->] (A) edge node[left]{$\cong$} (E) (A) edge (B) (B) edge node[above]{$(A)$} (C) (B) edge node[left]{$(B)$} (D) (D) edge node[left]{$(C)$} (F) (E) edge (F);
\end{tikzpicture}
\]
The left vertical map is an isomorphism integrally.
We first show that the other vertical maps in the rectangle are rational isomorphisms and postpone discussion of the map $(A)$ until afterwards.

By \cite[Corollary 1.2]{MadsenRothenbergEquivAut}, $TOP_{C_p}/PL_{C_p}$ is rationally trivial so the map $(C)$ is rationally a isomorphism.
Write $V=W\oplus\R^i$ where $W$ is a free representation.
If $V\to V'$ is an inclusion of $V$ into a larger representation, then $PL_{C_p}V\to PL_{C_p}V'$ is $(i+1)$-connected by \cite[Theorem 4.2]{MadsenRothenberg1}.
Since $PL_{C_p}=\varinjlim_V PL_{C_p}(V)$ is a filtered colimit, this shows that, if we take $V$ sufficiently large, the map $(B)$ is an isomorphism.

We now describe the map $(A)$.
Let $V$ be as above.
By the discussion before \cite[Lemma 2.1]{MadsenRothenberg2} and by \cite[Proposition 2.7]{MadsenRothenberg2}, there is an $i$-connected map
\[
PL_{C_p}(V)\simeq\widetilde{PL}_{C_p}(SW)\times PL(V^{C_p}).
\]
By $PL$-smoothing theory, the homotopy groups of $PL(V^{C_p})$ are rationally isomorphic to those of $O(V^{C_p})$ in degrees below $i$.
By \cite[Theorem 1]{CappellWeinbergerSimpleAS}, there are isomorphisms $\pi_{2m}(B\widetilde{PL}_{C_p}(SW))_{(0)}\cong \pi_{2m}B\widetilde{SPL}(SW)_{(0)}\times L_{2m}(C_p)_{(0)}$ when $W$ is sufficiently large.
The $L$-group is rationally isomorphic to $\pi_{2m}BU^{\times\frac{p-1}{2}}$ and $\pi_{2m}B\widetilde{SPL}(SW)$ is rationally isomorphic to $\pi_{2m}BO$ (this follows from the surgery exact sequence for the sphere).
It follows that there is a linear isomorphism $(A)$ claimed above.

To prove the proposition, we study the top row of the diagram above and carefully unravel the definition of $(A)$.
The composite
\[
 O_{C_p}(W)\times O(V^{C_p})\cong O_{C_p}(V)\rightarrow PL_{C_p}(V)\rightarrow \widetilde{PL}_{C_p}(SW)\times PL(V^{C_p})
\]
sends a pair $(\phii,\psi)$ to $(\phii|_{SW},\psi)$ where the first coordinate is the restriction to the unit sphere and the second coordinate is obtained by regarding $\psi$ as a $PL$-homeomorphism.\footnote{Technically, we should factor this through the space of piecewise smooth homeomorphisms.}
Rationally, the map $ O(V^{C_p})\rightarrow PL(V^{C_p})$ is an isomorphism on homotopy in degrees below $i$.
This isomorphism yields the desired isomorphism $b$.

We now examine the map $O_{C_p}(W)\rightarrow\widetilde{PL}_{C_p}(SW)$.
This map is studied in \cite{CappellWeinbergerSimpleAS} and \cite{WangChern}.
Details of the following arguments may be found in these references.
Cappell--Weinberger show that there is a rational equivalence
\[
B\widetilde{PL}_{C_p}(SW)\rightarrow B\widetilde{SPL}(SW)\times\tilde{L}(C_p)_{(0)}.
\]
We analyze the maps into each component independently.

The composition
\[
\pi_{2m}BO_{C_p}(W)\rightarrow\pi_{2m}B\widetilde{SPL}_{C_p}(SW)\rightarrow\pi_{2m}\tilde{L}(C_p)_{(0)}\cong\Q^{\frac{p-1}{2}}
\]
sends a $C_p$-vector bundle $E$ over $S^{2m}$ to the $G$-signature of $E$.
By the Atiyah--Singer $G$-signature theorem, this map may be identified with the linear transformation $\Phi_m$.

The composition
\[
\pi_{2m}BO_{C_p}(W)\rightarrow \pi_{2m}B\widetilde{SPL}_{C_p}(SW)\rightarrow \pi_{2m} B\widetilde{SPL}(SW)
\]
sends a $C_p$-vector bundle $E$ to its (non-equivariant) sphere bundle considered as a $PL$-block bundle and is rationally an isomorphism for $W$ sufficiently large (again by $PL$-smoothing theory).
This describes the map $a$ in the proposition.
The fact that $a$ is a surjection follows from considering the relation between Pontryagin classes and Chern classes.
\end{proof}

\begin{proof}[Proof of Theorem \ref{thm: main odd}]
By \ref{prop: reduction to fixed set}, it suffices to consider the long exact sequence of the fibration
\[
TOP_{C_p}/O_{C_p}\to BO_{C_p}\to BTOP_{C_p}.
\]
Since the homotopy groups of $BO_{C_p}$ and $BTOP_{C_p}$ are rationally trivial in odd degrees, we obtain the following exact sequence (where we have suppressed rationalization in the notation)
\[
0\to\pi_{2m}TOP_{C_p}/O_{C_p}\to\pi_{2m}BO_{C_p}\to\pi_{2m}BTOP_{C_p}\to\pi_{2m-1}TOP_{C_p}/O_{C_p}\to0.
\]
By applying Proposition \ref{prop: LES for p odd}, we obtain the following computation
\[
(\pi_m TOP_{C_p}/ O_{C_p})_{(0)}\cong\begin{cases}\mc{E}_p&m=1,2\\ \Q&m\equiv 3\text{ mod }4\\0&\text{ otherwise}\end{cases}
\]
which proves the theorem.
\end{proof}

We record a corollary of Proposition \ref{prop: LES for p odd} which will be useful in future work.
Here, we do not consider the rational homotopy groups.

\begin{cor}\label{cor: ker BO->BTOP trivial}
Suppose that either $m\neq 2$ or that $2$ has even order in $\F_p^{\times}$.
Then the composition
\[
\pi_m BU^{\times\frac{p-1}{2}}\to \pi_m BO_{C_p}\to \pi_m BTOP_{C_p}
\]
is injective.
\end{cor}
\begin{proof}
It suffices to consider the case where $m$ is even.
Under the hypotheses of the corollary, $\Phi_m$ is an injective $\Q$-linear transformation.
The result follows from Proposition \ref{prop: LES for p odd} and the fact that $\pi_m BU$ is torsion-free.
\end{proof}

\section{The Case $p=2$}
In the case $G=C_2$, Madsen--Rothenberg \cite{MadsenRothenberg2} show that there is no equivariant transversality.
The technical result behind this statement is that the maps
\[
\mc{S}_{k+n}(\RP^n)\rightarrow \mc{S}_{k+n+1}(\RP^{n+1})
\] 
do not eventually become isomorphisms.
Here, we use $\mc{S}$ to denote the surgery theoretic structure set.
The map can be written explicitly as follows.
If $f:Y\rightarrow\RP^n\times D^k$ represents an element of $\mc{S}_{k+n}(\RP^n)$ let $\tilde{f}:\tilde{Y}\rightarrow S^n\times D^k$ be the $C_2$-equivariant map on the universal cover.
Then the class of $f$ is sent to $(\tilde{f}*\op{Id}_{S^0})_{C_2}:(\tilde{Y}*S^0)_{C_2}\rightarrow (S^n * S^0\times D^k)_{C_2}$ where $C_2$ acts nontrivially on $S^0$.

Although these maps are not isomorphisms, the colimit can be computed following \cite[Section 4]{MadsenRothenberg2}.
\begin{prop}\label{prop: stable structure set of RP^n}
There is an isomorphism
\[
\varinjlim_{n}\mc{S}_{n+k}^{PL}(\RP^n)\cong\begin{cases}\Z\oplus\bigoplus\Z/2&k\equiv 0\text{ mod }4\\\bigoplus\Z/2&\text{otherwise}\end{cases}
\]
where the torsion part is a countable direct sum.
\end{prop}
\begin{proof}
In the proof of \cite[Proposition 4.3]{MadsenRothenberg2}, Madsen--Rothenberg show that there are the following commuting diagrams.
\[
\begin{tikzpicture}[scale=2]
\node (A) at (0,1) {$\mc{S}_{2\ell+2m+1}^{PL}(\RP^{2m+1})$};\node (B) at (3,1) {$[S^{2\ell}\wedge\RP^{2m+1}_+,F/PL]$};
\node (C) at (0,0) {$\mc{S}_{2\ell+2m+5}^{PL}(\RP^{2m+5})$};\node (D) at (3,0) {$[S^{2\ell}\wedge\RP^{2m+5}_+,F/PL]$};
\path[->] (A) edge node[above]{$\cong$} (B) (A) edge (C) (C) edge node[above]{$\cong$} (D) (D) edge (B);
\end{tikzpicture}
\]
The left vertical map is the composite of suspensions and the right vertical map is surjective.
The horizontal maps are obtained from the $PL$-surgery exact sequence.
Computing the cohomology group and using that the maps on the left form a cofinal system proves the proposition in the case $k$ is even.
The case where $k$ is odd follows from \cite[Proposition 4.3]{MadsenRothenberg2}.
\end{proof}

\begin{prop}\label{prop: rational homotopy of PL_C_2}
For $k>0$, there are isomorphisms
\[
\pi_k PL_{C_2}\otimes \Q\cong\pi_k TOP_{C_2}\otimes\Q\cong\begin{cases}\Q^2&k\equiv 0\text{ mod }4\\0&\text{ otherwise}\end{cases}.
\]
\end{prop}
\begin{proof}
Let $W$ be a free $C_2$-representation of dimension $n+1$ with unit sphere $SW$.
Then,
\[
\pi_k F_{C_2}(SW)/\widetilde{PL}_{C_2}(SW)\cong\mc{S}^{PL}_k(\RP^n)
\]
for $k>0$.
Since $F_{C_2}(SW)$ has finite homotopy groups, there is an isomorphism
\[
\pi_k B\widetilde{PL}_{C_2}(SW)\otimes\Q\cong\mc{S}^{PL}_k(\RP^n)\otimes\Q.
\]

If $V\cong W\oplus\R^m$, then $\pi_k PL_{C_2}(V)\cong\pi_k\widetilde{PL}_{C_2}(SW)\times \pi_kPL(\R^m)$ for $k\le m$ (see \cite[Section 2]{MadsenRothenberg2}).
Taking a colimit over representations $V$ and using the identification above proves the result for $PL_{C_2}$.

The case of $TOP_{C_2}$ follows similarly; by mapping the $PL$-surgery sequence to the $TOP$-surgery sequence, one sees that the $TOP$ version of the diagram in proof of Proposition \ref{prop: stable structure set of RP^n} is the same rationally.
\end{proof}

Recall the surgery groups for $C_2$ with the trivial orientation are as follows.
\[
L_k(C_2)\cong\begin{cases}\Z\oplus\Z&k\equiv0\text{ mod }4\\0&k\equiv1\text{ mod }4\\\Z/2&k\equiv 2,3\text{ mod 4}\end{cases}
\]
As we will be concerned with the colimit $\varinjlim_m\mc{S}_{k+m+1}(\RP^m)$, we will assume $m$ is odd so only the trivial orientation will be relevant.
A structure on $\RP^m\times D^k$ pulls back to a structure on $S^m\times D^k$.
Following \cite{CappellWeinbergerSimpleAS}, this yields a map of structure spaces $\mc{S}(\RP^m)\rightarrow\mc{S}(S^m)$ whose fiber is the reduced $L$-space $\tilde{L}_{m+1}(C_2)$ away from $2$.

Consider the following segment in the long exact sequence of homotopy groups.
\[
\tilde{L}_{m+k+1}(C_2)\rightarrow\mc{S}_{m+k+1}(\RP^m)\rightarrow\mc{S}_{m+k+1}(S^m)\rightarrow\tilde{L}_{m+k}(C_2)
\]
When $k\equiv0\text{ mod }4$ and $m\equiv1\text{ mod }4$, the surgery groups vanish away from $2$.
This is also true when $k\equiv2\text{ mod }4$ and $m\equiv 3\text{ mod }4$.
By fixing $k$ and choosing the appropriate cofinal system of $\RP^m$, we see that
\[
\varinjlim_m\mc{S}_{k+m+1}(\RP^m)\cong\varinjlim_m\mc{S}_{k+m+1}(S^m)
\]
is rationally an isomorphism.
Identifying $\mc{S}(M)\simeq F(M)/\widetilde{TOP}(M)$, we conclude that the map
\[
\pi_k\varinjlim_mB\widetilde{TOP}(\RP^m)\otimes\Q\rightarrow\pi_k\varinjlim_m B\widetilde{TOP}(S^m)
\]
induced by pulling back a block homeomorphism is an isomorphism.

\begin{theorem}\label{thm: main thm C_2}
There is an isomorphism
\[
\pi_k TOP_{C_2}/O_{C_2}\otimes\Q\cong 0.
\]
\end{theorem}
\begin{proof}
We show that the map $BO_{C_2}\rightarrow BTOP_{C_2}$ induces an isomorphism on the rational homotopy groups.
Let $V\cong W\oplus\R^n$ be a $C_2$-representation where $W$ is a direct sum of sign representations and $C_2$ acts trivially on $\R^n$.
Then $BO_{C_2}(V)\cong BO(W)\times BO(\R^n)$ and there is a map
\[
BTOP_{C_2}(V)\rightarrow B\widetilde{TOP}_{C_2}(SW)\times BTOP(\R^n)
\]
which induces isomorphisms on $\pi_k$ for $k\le n$.

By considering the composite $BO(W)\times BO(\R^n)\rightarrow B\widetilde{TOP}_{C_2}(SW)\times BTOP(\R^n)$, it suffices to show that the map $BO(W)\rightarrow B\widetilde{TOP}_{C_2}(SW)$ induces an isomorphism on rational homotopy groups for a cofinal family of representations $W$.
Assume $k\equiv0\text{ mod }4$ and take our cofinal family to be $(m+1)$-copies of the sign representation where $m\equiv1\text{ mod }4$.
Identifying $\widetilde{TOP}_{C_2}(SW)$ with $\widetilde{TOP}(\RP^m)$, the remarks above show that
\[
\pi_kB\widetilde{TOP}(\RP^m)\rightarrow\pi_kB\widetilde{TOP}(S^m)
\]
is rationally an isomorhpism.

It follows that
\[
\pi_kB\widetilde{TOP}_{C_2}(W)\otimes\Q\cong\pi_kB\widetilde{TOP}(\R^{m+1})\otimes\Q
\]
where the map is induced by forgetting the group action.
Since the composite $BO(W)\rightarrow B\widetilde{TOP}(\R^{m+1})$ induces an isomorphism on rational homotopy groups in dimensions at most $2\dim W$, this proves the theorem for $k\equiv0\text{ mod }4$.
The case $k\equiv2\text{ mod }4$ is similar.
\end{proof}

Proposition \ref{prop: reduction to fixed set} and Theorem \ref{thm: main thm C_2} prove Theorem \ref{thm: main even}.

\section{``Topological Invariance'' of Rational Chern Classes}\label{sec: chern classes}
The finiteness of the homotopy groups $\pi_kTOP/O$ implies that $BTOP$ and $BO$ are rationally homotopy equivalent.
It follows that the rational Pontryagin classes of a smooth manifold depend only on the underlying topological manifold.
We apply arguments analogous to those above for the group $G=C_4$ to show that the rational Chern classes of a complex vector bundle depend only on the underlying $C_4$-equivariant topological micro-bundle.
We regard $\C$ as a $C_4$-representation with the generator acting via multiplication by $i$.
Thus, $U=\bigcup_n U(n)=\bigcup_n O_{C_4}(\C^n)$.
Define $TOP^U(n):=TOP_{C_4}(\C^n)$ and $TOP^U:=\bigcup_n TOP^U(n)$.

\begin{definition}\label{def: top almost complex structure}
A \emph{topological almost complex manifold} is a topological manifold $M^{2n}$ with an automorphism of the tangent microbundle $J:\tau M\to\tau M$ such that the following hold:
\begin{itemize}
\item $J^4=\op{Id}$,
\item For each $x\in M$, the fiber is equivariantly homeomorphic to $\C^n$.
\end{itemize}
\end{definition}

Recall that an almost complex manifold is a smooth manifold $M^{2n}$ with a complex reduction of its tangent bundle.
Similarly, a topological almost complex manifold $M^{2n}$ may be regarded as a topological manifold with a $TOP^U(n)$ reduction of its tangent microbundle.
Every almost complex manifold has an underlying topological almost complex manifold.
In this section, we prove the following.

\begin{theorem}\label{thm: BU BTOP U}
The map $BU\to BTOP^U$ induces a surjection on rational cohomology.
\end{theorem}
\begin{proof}
We are interested in the stabilization of the diagram
\[
\begin{tikzpicture}[scale=2]
\node (A) at (0,1) {$BU(n)$};\node (B) at (2,1) {$BU(n)\times BO(\ell)$};
\node (C) at (0,0) {$BTOP^U(n)$};\node (D) at (2,0) {$BTOP_{C_4}(\C^n\oplus\R^\ell)$};
\path[->] (A) edge (B) (A) edge (C) (B) edge (D) (C) edge (D);
\end{tikzpicture}.
\]
Let $BTOP^{U'}:=\varinjlim_{n,\ell}BTOP_{C_4}(\C^n\oplus\R^{\ell})$ and let $B\widetilde{TOP}^{U}:=\varinjlim_{n}B\widetilde{TOP}_{C_4}(\C^n)$.
It suffices to show that the induced map
\[
H^*(BTOP^{U'};\Q)\to H^*(BU;\Q)
\]
is surjective.
As before, there are maps
\[
BTOP_{C_4}(\C^n\oplus\R^{\ell})\to BTOP_{C_4}(\C^n\oplus\R^{\ell},\R^{\ell})\times BTOP(\R^{\ell})\to B\widetilde{TOP}_{C_4}(S(\C^n))\times BTOP(\R^{\ell}).
\]

As in the odd order case, the map
\[
B\widetilde{TOP}(S(\C^n)/C_4)\to B\widetilde{TOP}(S(\C^n))\times \tilde{L}(C_4)_{(0)}
\]
induces an isomorphism on rational homotopy groups.
The composition $BU(n)\to\tilde{L}(C_4)_{(0)}$ may be described by the Atiyah--Singer $G$-signature theorem.
Specifically, if $E\in\pi_{2m}BU(n)$ is a complex vector bundle over $S^{2m}$, the corresponding element in $\pi_{2m}\tilde{L}(C_4)_{(0)}\cong\Q$ is $\Phi_mc_m(E)$ where $\Phi_m\in\Q$ or $i\Q$ (according to the parity of $m$).
By \cite[Corollary 3.5]{EwingSemifree} and \cite[Proposition 3.6]{EwingSemifree} the coefficients $\Phi_m$ are nonzero.

Stabilizing, $BU\to BTOP^{U'}$ may be identified with a map $BU\to B\widetilde{TOP}^U\times BTOP$.
The analysis of the Atiyah--Singer formula above shows that the composition $BU\to B\widetilde{TOP}^U\times BTOP\to \tilde{L}(C_4)$ is an isomorphism on rational homotopy groups, which suffices to prove the theorem.
\end{proof}

\begin{cor}\label{cor: Top invariance of Chern}
Suppose $M$ is a topological almost complex manifold.
Then every almost complex manifold with underlying topological almost complex manifold $M$ has the same rational Chern classes.
\end{cor}
\begin{proof}
This follows from Theorem \ref{thm: BU BTOP U} and from considering the diagram
\[
\begin{tikzpicture}[scale=2]
\node (A) at (0,0) {$M$};\node (B) at (2,1) {$BU$};\node (C) at (2,0) {$BTOP^U$};
\path[->] (A) edge (B) (A) edge (C) (B) edge (C);
\end{tikzpicture}.
\]
\end{proof}

Without the allowing ourselves to stabilize by adding trivial representations, it is difficult to relate the fiberwise classifying space to the block classifying space.
We end with the following question.

\begin{question}
What is the fiber of $BTOP^U\to B\widetilde{TOP}^U$?
\end{question}
\bibliographystyle{alpha}
\bibliography{StableZpSmoothing.bib}
\end{document}